\documentclass{amsart}

\usepackage{lmodern,url,enumerate,mathtools,microtype}
\usepackage[hmargin=1.8in,vmargin=1in]{geometry}
\usepackage{graphicx}
\usepackage{color}

\usepackage{amsfonts,latexsym,amsthm,amssymb,amsmath,amscd,euscript,tikz}
\usepackage{comment}
\usepackage{amsmath}
\usepackage{mathtools}
\usepackage{mdframed}
\usepackage{enumerate}
\usepackage{multicol}        			
\usepackage{color}           			
\usepackage{array}
\usepackage{ae}
\usepackage{accents}
\usepackage{epsfig}
\usepackage[e]{esvect}

\usepackage{yfonts}

\usepackage{empheq}
\usepackage{verbatim}
\usepackage{hyperref}
\hypersetup{colorlinks=true,citecolor=blue,linkcolor = blue,urlcolor =black,linkbordercolor={1 0 0}}
\usepackage{float}
\usepackage[latin1]{inputenc}

\usepackage[alphabetic, msc-links]{amsrefs}

\definecolor{r}{rgb}{.9,0.1,.3}

\definecolor{verdeosc}{rgb}{0,0.6,0}

\newcommand{\Rr}{\mathcal R}

 \newcommand{\RR}{\mathbf{R}}  
    
  \newcommand{\Div}{\operatorname{Div}}

 \newcommand{\area}{\operatorname{area}}
 \newcommand{\eps}{\epsilon}

\newcommand{\ee}{\mathbf e}
\newcommand{\on}{\operatorname}

\newcommand{\mres}{\mathbin{\vrule height 1.6ex depth 0pt width
0.13ex\vrule height 0.13ex depth 0pt width 1.3ex}}

\newcommand{\pdf}[2]{\frac{\partial #1}{\partial #2}}

\newtheorem*{theorem*}{Theorem}
\newtheorem{theorem}{Theorem}

\newtheorem{lemma}[theorem]{Lemma}

\newtheorem{proposition}[theorem]{Proposition}

\newtheorem{claim}{Claim}
\newtheorem*{claim*}{Claim}
\theoremstyle{definition}
\newtheorem{remark}[theorem]{Remark}

\def\pproof#1{\@ifnextchar[\opargproof
{\opargproof[\it Proof of #1.]}}
\def\opargproof[#1]{\par\noindent {\bf #1 }}

\makeatother

\newcommand{\cP}{\mathcal{P}}

\newcommand{\Infty}{\widehat{\infty}}
\newcommand{\graph}{\operatorname{graph}}

\definecolor{ForestGreen}{RGB}{34,139,34}

\title[Classification of Semigraphical Translators]{Classification of Semigraphical Translators}
\thanks{Some of this work was carried out while the authors were visitors at the Simons Laufer Mathematical Sciences Institute (SLMath, formerly MSRI) in Berkeley, CA during the Fall 2024 semester, in a program supported by National Science Foundation Grant No. DMS-1928930.}

\author[F. Mart\'in]{Francisco Mart\'{\i}n}

\address{Francisco Mart\'in\newline
Departamento de Geometr\'ia y Topolog\'ia  \newline
Instituto de Matem\'aticas  de Granada (IMAG) \newline
Universidad de Granada\newline
18071 Granada, Spain\newline
{\sl E-mail address:} {\bf fmartin@ugr.es}
}%
\thanks{F. Mart\'in was partially supported by the grants PID2020-116126-I00 and PID2024-156031NB-I00 funded by MICIU/AEI/
10.13039/501100011033 and by the IMAG-Maria de Maeztu grant CEX2020-001105-M funded by MICIU/AEI/10.13039/501100011033.}
\author[M. S\'aez]{Mariel S\'aez}

\address{Mariel S\'aez\newline
P. Universidad Cat\'olica de Chile \newline
Facultad de Matem\'aticas\newline
Avda. Vicu\~na Mackenna 4860 \newline
Macul, Santiago, 6904441, Chile \newline
{\sl E-mail address:} {\bf mariel@uc.cl}
}
\thanks{ M. S\'aez was partially supported by the grant Fondecyt Regular 1190388. }

\author[R. Tsiamis]{Raphael Tsiamis}

\address{Raphael Tsiamis\newline
Department of Mathematics \newline
Columbia University \newline 
2990 Broadway, New York NY 10027, USA \newline
{\sl E-mail address:} {\bf r.tsiamis@columbia.edu}
}
\thanks{ R. Tsiamis was partially supported by the A.G. Leventis Foundation Scholarship and the Onassis Foundation Scholarship. 
He acknowledges the hospitality of the University of Granada, where part of this work was carried out.}
\author[B. White]{Brian White}

\address{Brian White\newline
Department of Mathematics \newline
 Stanford University \newline 
  Stanford, CA 94305, USA\newline
{\sl E-mail address:} {\bf bcwhite@stanford.edu}
}
\date{November 22, 2024. Revised February 16,  2026}

\begin{document}

\begin{abstract}
We complete the classification of semigraphical translators for mean curvature flow in $\mathbf{R}^3$ that was initiated
by Hoffman-Mart\'in-White. 
Specifically, we show that there is no solution to the translator equation on the upper half-plane with alternating positive and negative infinite boundary values,
and we prove the uniqueness of pitchfork and helicoid translators.
The proofs use Morse-Rad\'o theory for translators and an angular maximum principle.
\end{abstract}

\maketitle
\tableofcontents
\vspace*{0.15in}


\section{Introduction}

A surface $M$ in $\RR^3$ is called a 
{\bf translator} (for mean curvature flow) with velocity 
  $\mathbf{v}$ if
  \begin{equation*}\label{eqn:translator-MCF}
M \mapsto M_t := M + t \mathbf{v}
\end{equation*} 
is a mean curvature flow, i.e., if the normal velocity at each point is
equal to the mean curvature vector: 
$\vec{H}=\mathbf{v}^\perp$. By rotating and scaling,
it suffices (for nonzero velocities)
to consider the case $\mathbf{v}= -\ee_3$.
In that case, we refer to $M$ simply as a {\bf translator}.  
Ilmanen~\cite{ilmanen-1994} observed that $M$ is a translator
if and only if it is minimal with respect to the translator metric $g_{ij}(x,y,z)=e^{-z}\delta_{ij}$.

If the surface $M$ is the graph of a function $u$ over a region of the $(x,y)$-plane, then
the translator equation $\vec{H}=(-\ee_3)^\perp$ is equivalent to
\begin{equation} \label{eqn:translatorequationR3}
    \Div \left( \frac{D u}{\sqrt{1+|D u|^2}} \right)= - \frac{1}{\sqrt{1+|D u|^2}},
\end{equation}
which is a quasilinear elliptic PDE.

A smooth, connected, complete properly embedded translator $M$ is called {\bf semigraphical} if it contains a non-empty, discrete collection $\{ L_n \}_n$ of vertical lines such that $M \setminus \bigcup_n L_n$ is a graph~$z=z(x,y)$.
This notion  was introduced 
 by Hoffman-Mart\'in-White
 in \cite{tridents}, which included 
  a partial classification. 
Building on their work, we complete the classification of semigraphical translators and prove 
\begin{theorem}\label{main-theorem}
A semigraphical translator $M$ in $\RR^3$ is one of the following:
\begin{enumerate}[(1)]
\item a (doubly periodic) Scherk translator;
\item a (singly periodic) Scherkenoid;
\item a (singly periodic) 
  translating helicoid;
\item a pitchfork; 
\item a (singly periodic) trident.
\end{enumerate}
Up to isometries of $\RR^3$, 
\begin{enumerate}[\upshape (i)]
\item 
translating helicoids and 
 pitchforks 
are uniquely determined by 
a single parameter, the width of a
fundamental region,
\item
tridents are uniquely determined by 
a single parameter, the period, and
\item
Scherk translators and Scherkenoids are uniquely determined by two parameters, the width and an angle.
\end{enumerate}
\end{theorem}

See~\cite{tridents} for pictures of the surfaces in Theorem~\ref{main-theorem}.

Theorem~\ref{main-theorem}, minus Assertions (i), (ii), and (iii), is the same as Theorem 34 of~\cite{tridents}, except that the list in Theorem 34 included a sixth item, sometimes referred to as a ``yeti" (or ``abominable flowman").
In~\cite{tridents}, yetis were conjectured not to exist.
Assertions (i), (ii), and (iii) assert that, for the allowed values of the parameters, the indicated surfaces exist and are unique.
Existence and uniqueness for the tridents was proved in~\cite{tridents}, and existence  and uniqueness for the other surfaces was proved in~\cite{scherk-like}, except that uniqueness of pitchforks and uniqueness of translating helicoids were left as conjectures.

In this paper, we complete the proof of Theorem~\ref{main-theorem} by proving those three conjectures: uniqueness for pitchforks (\S\ref{pitchfork-section}) and translating helicoids (\S\ref{helicoid-section}), and non-existence of the yeti (\S\ref{yeti-section}).

It is natural to study and classify semigraphical translators following the classification of graphical solitons in $\RR^3$, due to the work of Wang, Spruck-Xiao, and Hoffman-Ilmanen-Mart\'in-White \cites{ancient-wang, spruck-xiao, graphs-himw}.
Another starting point for the study of semigraphical translators is the construction of translating analogs to Scherk's doubly periodic minimal surfaces.
These were obtained in \cite{scherk-like} along with more examples, some resembling well-known minimal surfaces and some without such analogs.
For example, the pitchfork~(4) contains a single vertical line of symmetry (see Figure~\ref{fig:pitchfork}).
The uniqueness of pitchforks can therefore be viewed as the translator analog of the characterization of the half Enneper's surface as the unique properly embedded non-flat stable minimal surface with straight line boundary and quadratic area growth, cf. \cites{perez}.

The doubly periodic (``Scherk'') and singly periodic (``Scherkenoid'') translating surfaces (1) and (2) of Theorem \ref{main-theorem} are constructed in \cite{scherk-like} via repeated Schwarz reflections around vertical axes of the graph of a function solving the translator
 equation~\eqref{eqn:translatorequationR3} on:
\begin{enumerate}[(i)]
\item for Scherk translators, a parallelogram $\cP(\alpha, w,L)$ of base angle $\alpha \in (0,\pi)$, width $w$, and unique length $L(\alpha,w)$, with boundary values $\pm \infty$ on the pairs of opposite sides; 
\item for Scherkenoids, the half-strip of boundary data that results from extending the horizontal sides of $\cP(\alpha,w,L)$ above into rays in the $(+1,0)$-direction.
\end{enumerate}
Using a sequential compactness theorem \cite{scherk-like}*{Theorem 10.1}, one can extract subsequential limits from sequences of these Scherk or Scherkenoid surfaces as the base angle $\alpha$ tends to $\pi$. 
The limit surface is called a {\bf translating helicoid} if $w<\pi$ and a {\bf pitchfork} if $w \ge \pi$.
The uniqueness results in the current paper show that the limit is unique and thus does not depend on the choice of subsequence.

Pitchforks are simply connected and have entropy $3$.
They might therefore arise as blowups of mean curvature flows of surfaces that are initially closed and smoothly immersed.
(By recent work~\cite{bamler-kleiner} of Bamler and Kleiner, they cannot arise as blowups if the initial surface is embedded.)
All semigraphical translators other than pitchforks and tridents have cubic area growth, and therefore infinite entropy, hence they cannot arise as such blowups.
Tridents have finite entropy but infinite genus, and thus they cannot arise as blowups at the first singular time.

In \cite{gama}, pitchforks are characterized as the only complete embedded simply connected translating solitons that are contained in a slab, have entropy $3$, and contain a vertical line. 

See Figure \ref{fig:pitchfork} and Figure \ref{fig:helicoid} for example illustrations of these translators and their fundamental pieces.
\begin{figure}[ht]
{\includegraphics[width=.45\textwidth]{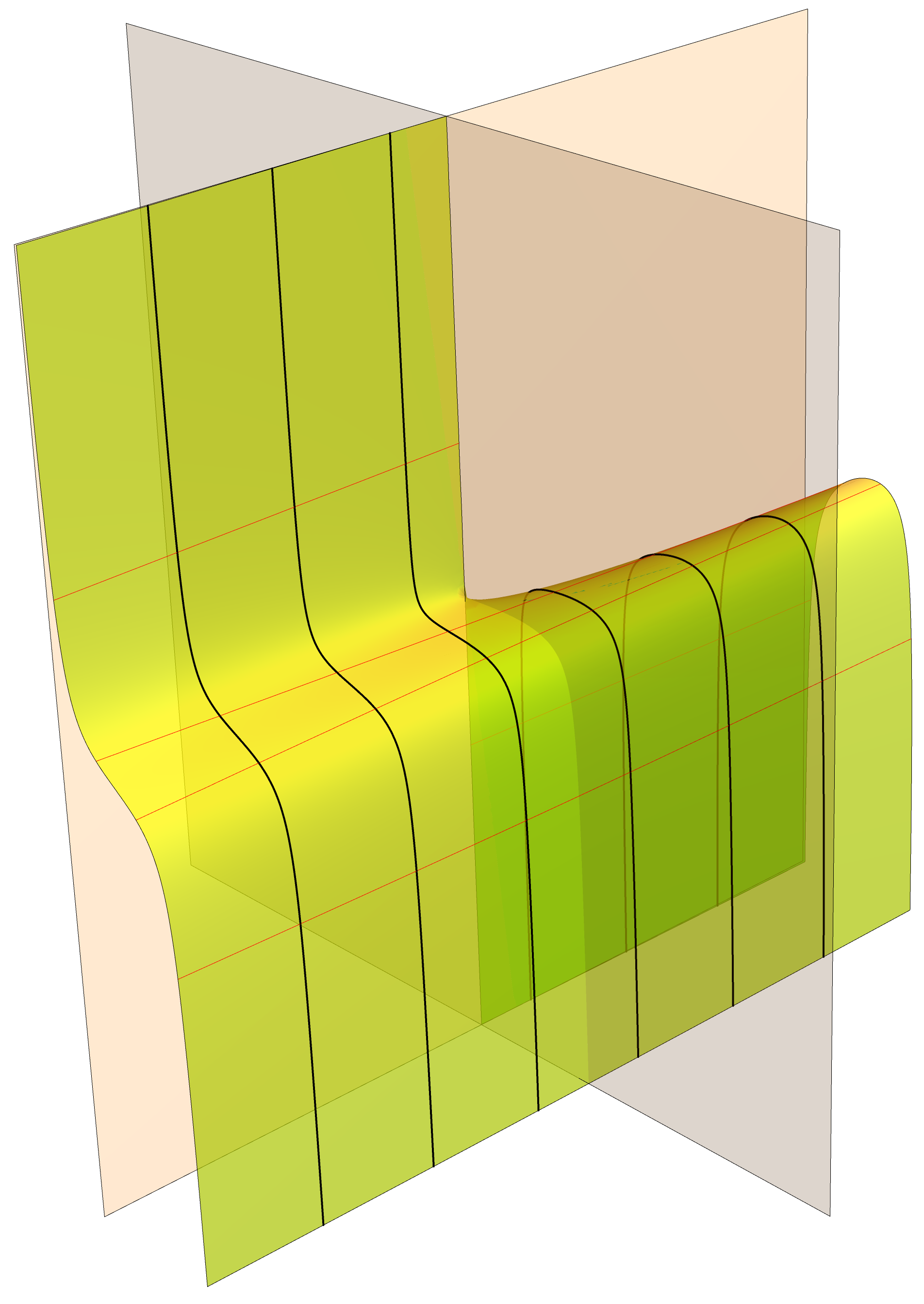}}
{\includegraphics[width=.45\textwidth]{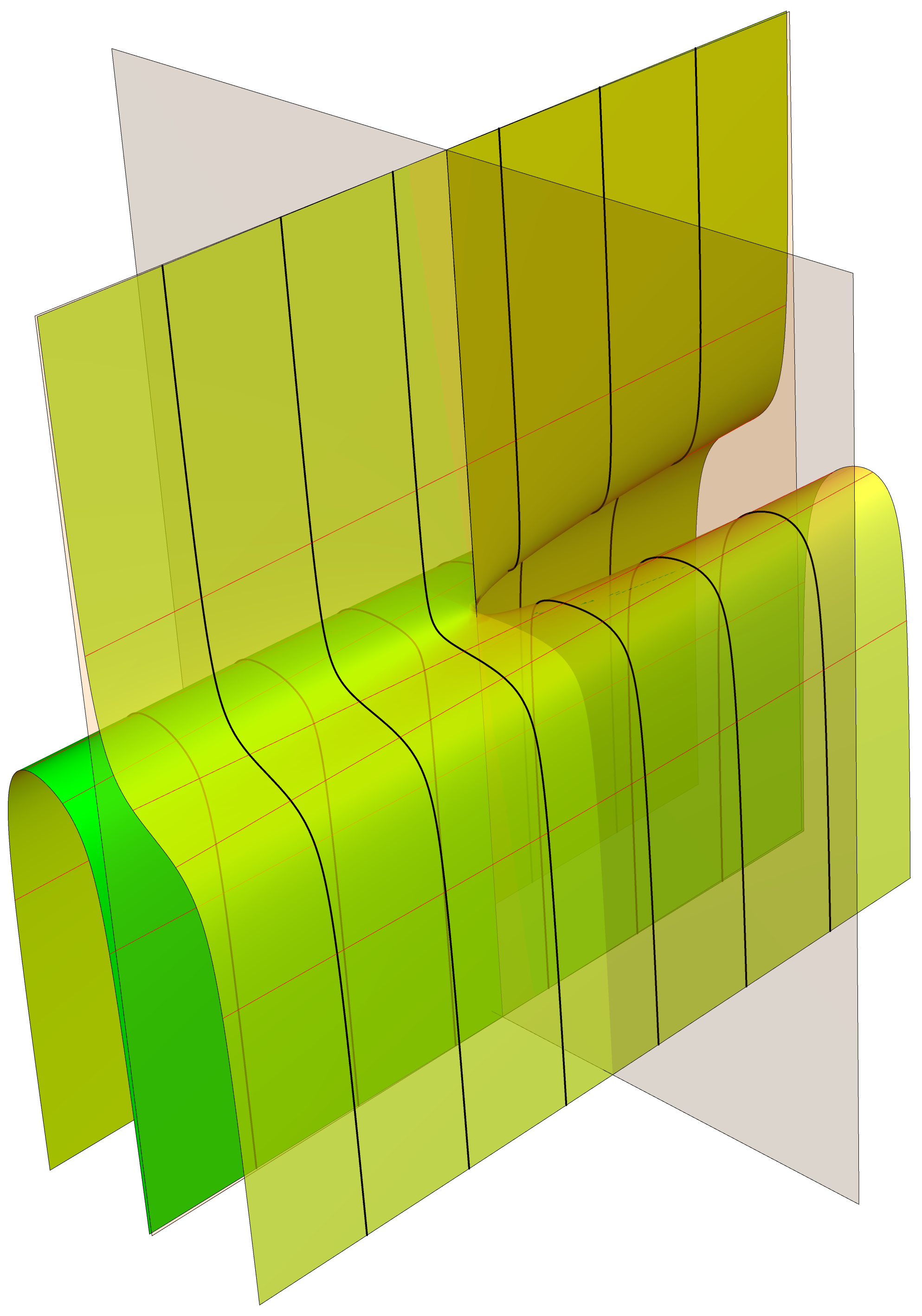}}
\caption{\small \sffamily Left: A fundamental piece of the pitchfork of width $\pi$.
Right: The whole surface, obtained from the fundamental piece by a $180^{\circ}$ rotation around the $z$-axis.} 
\label{fig:pitchfork}
\end{figure}
\begin{figure}[ht]
{\includegraphics[width=.4\textwidth]{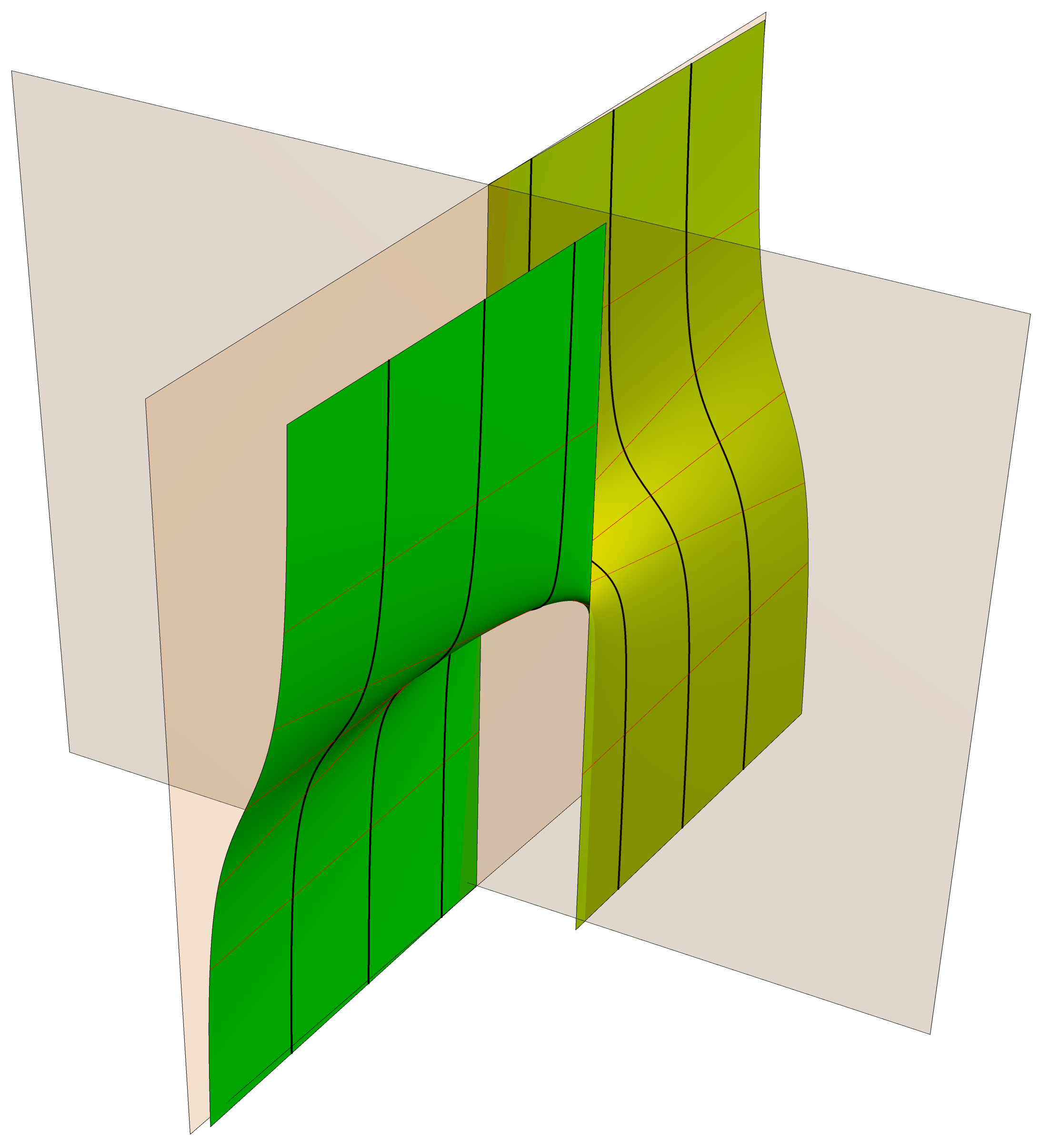}}
{\includegraphics[width=.45\textwidth]{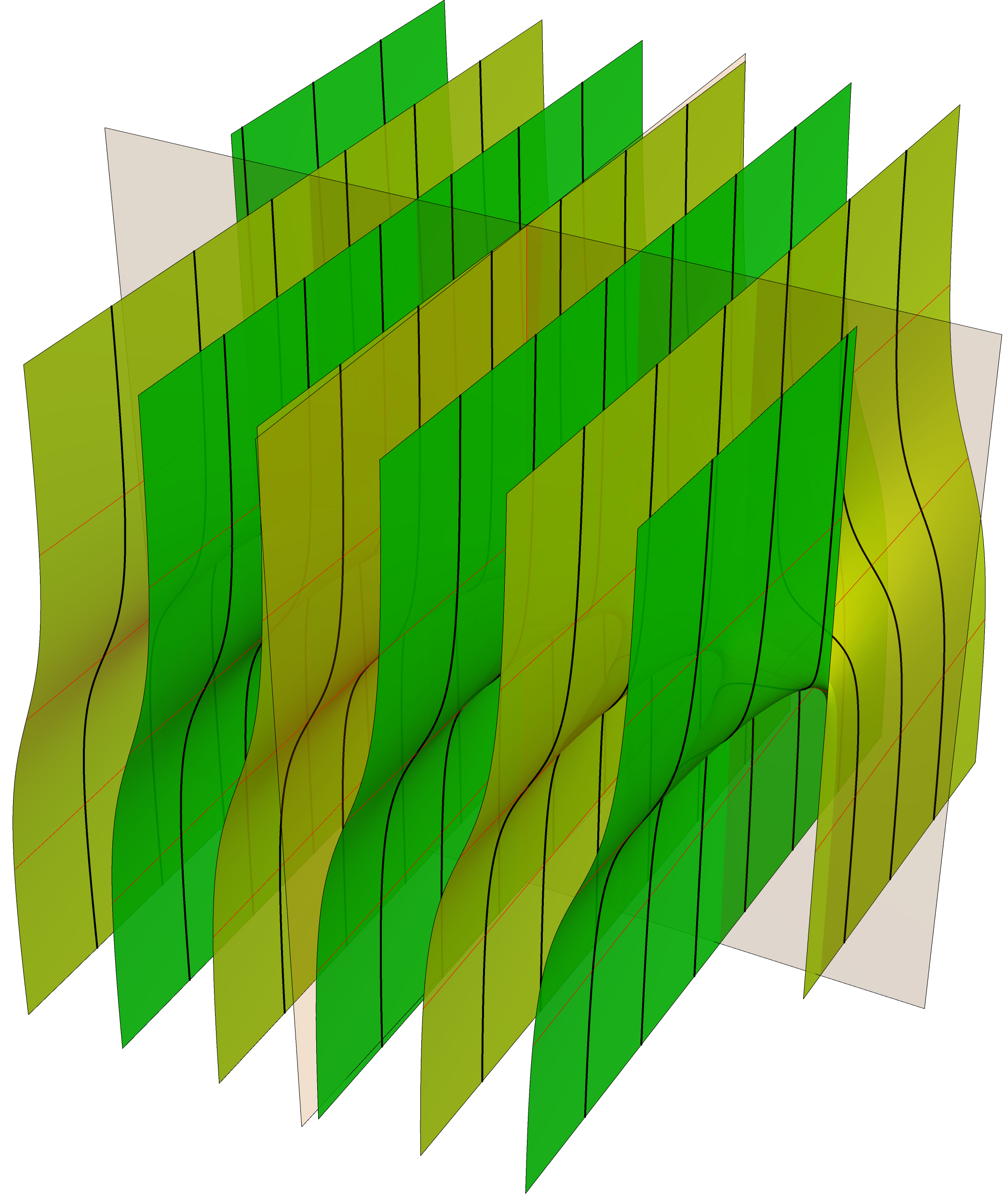}}
\caption{\small \sffamily Left: A fundamental piece of the helicoid of width $\pi/2$. Right: Part of the surface, obtained by successive reflections along the vertical boundary lines.}\label{fig:helicoid}
\end{figure}
\newline

{\bf Tridents}, the last item~(5)
in the list in Theorem~\ref{main-theorem}, were originally constructed for small neck sizes by Nguyen using gluing techniques~\cite{nguyen}.
Later, Hoffman-Mart\'in-White constructed tridents as semigraphical translators of arbitrary neck size $a > 0$~\cite{tridents}.
They also proved that, as $a\to\infty$, the tridents converge subsequentially to a pitchfork of width $\pi$.
They observed that if there is a unique pitchfork of width $\pi$, then one would get convergence, not merely subsequential convergence, and the limit would coincide with the pitchfork obtained as a limit of Scherkenoids of width $\pi$.
In this paper, we prove the uniqueness, and thus those consequences of uniqueness follow.

There are interesting translators that are neither graphical nor semigraphical.
Examples include:
helicoid-like translators that contain a vertical line and are invariant under a screw motion about that line~\cite{helicoidal},
surfaces obtained by gluing~\cites{davila, nguyen-reapers, nguyen-doubly, smith},
rotationally invariant annuli~\cite{clutterbuck}, and a large family~\cite{annuli} of annuli contained in slabs (and therefore not rotationally invariant).

Let us indicate the main ideas of the paper by examining the geometry of a potential yeti, in contrast with that of the pitchfork and the helicoid of Figures~\ref{fig:pitchfork} and~\ref{fig:helicoid}.
It is not hard to see that a ``yeti'' translator $\mathcal{Y}$, if it existed, would necessarily have very exotic properties. 
Notably, the Gauss map of $\mathcal{Y}$ restricted to any half-strip would come arbitrarily close to the equator of the upper hemisphere, and the gradient of the graphing function $u$ would be unbounded along every $x$-ray. Moreover, the entropy of $\mathcal{Y}$ would be either infinite or (by standard considerations, cf. \cite{gama-martin-moller-2025}) an integer; it is conjectured that the latter is impossible for non-collapsed translators.
These conditions would force very complicated geometric behavior, which we rule out using maximum principle techniques.
 
The arguments involved in proving Theorem~\ref{main-theorem} use Morse-Rad\'o theory~\cite{morse-rado}, which studies critical points and level sets for a class of functions arising in minimal surface theory.
(Recall that translators are minimal surfaces with respect to the translator metric.)
The uniqueness proofs use an angular maximum principle that may be of independent interest.
The proof of non-existence of the yeti involves comparison with a family of appropriately placed grim reaper surfaces, inspired by a technique first employed by Chini \cite{chini}.
This sliding-maximum principle allows us to derive a gradient bound for the solution $u$ in a half-space.
Then, an asymptotic limit of $\mathcal{Y}$ would be a complete graphical translator over the half-space, violating the classification Theorem of~\cite{graphs-himw}.

\subsection{Acknowledgments}

The authors thank David Hoffman for many helpful discussions and comments, as well as for his crucial role in the classification of semigraphical translators. 

\section{An Angular Maximum Principle}

In this section, we prove a version of the maximum principle that will be used to establish the uniqueness of pitchforks and helicoidal translators.

\begin{theorem}\label{angle-theorem}
Suppose that $\Omega$ is a domain contained in a half-strip 
\[
S:=[a,\infty)\times (b,c).
\]
Suppose that $u_1, u_2: \Omega\to \RR$
are solutions of the translator equation that extend continuously to $\partial\Omega$,
and suppose that
\[
(u_2-u_1)|_{\partial \Omega} = d
\]
for some constant $d$.
Suppose also that their normal vectors $\nu_i := \frac{(- D u_i , 1)}{\sqrt{1 + |D u_i|^2}}$ satisfy
\[
 \inf_{(x,y)\in \Omega \cap \{ x\ge x_0\}} \nu_i(x,y)\cdot \ee_2 > \eps > 0
\]
for some $x_0<\infty$ and $\eps$.
Then $u_2-u_1\equiv d$.
\end{theorem}

\begin{proof}
After translating the domain and replacing $u_2$ by $u_2 - d$, we can assume that $b=0$ and $d=0$.

Suppose, contrary to the theorem, that $u_2-u_1$ is not constant.
By relabelling, we may assume that there are points in $\Omega$
where $u_2>u_1$.

By the strong maximum principle,
\[
\mu:=\sup_{\overline{\Omega\cap \{x\le x_0\}}} (u_2-u_1) < \sup (u_2 - u_1).
\]
By replacing $\Omega$ by $\{(x,y)\in \Omega: u_2-u_1 > \mu\}$ and then replacing $u_2$ by $u_2 - \mu$, we can assume that $\Omega$ lies in 
\[
 [x_0,\infty)\times (0,c),
\]
that $u_2=u_1$ on $\partial \Omega$, and that $u_2\ge u_1$ on $\Omega$.
By the strong maximum principle, $u_2>u_1$ in $\Omega$.

By translating the above configuration in the $x$-direction, we can assume that $x_0$ is as large as we like.
In particular, we can assume that 
\[
 x_0 > \frac{c}{\eps}.
\]
Since $x_0>0$, we have well-defined polar coordinates $(r,\theta)$ on $\overline{\Omega}$, with $0 \le \theta<\pi/2$.
Moreover, $0 \leq r(p) \sin \theta(p) = y(p) \leq c$ for any $p \in \overline{\Omega}$.

Let $s_i=(1+|Du_i|^2)^{1/2}$.  Then
\begin{align*}
\pdf{u_i}{\theta}
&=
\pdf{u_i}{x}\pdf{x}{\theta} + \pdf{u_i}{y} \pdf{y}{\theta}
\\
&=
\pdf{u_i}{x} (-y) + \pdf{u_i}{y} x
\\
&= 
-  s_i \, \nu_i \cdot (-y, x,0)
\\
&=
  s_i ( y \,\nu_i\cdot \ee_1 - x\, \nu_i\cdot \ee_2)
\\
&\le
 s_i ( c - x_0 \eps)
\\
&<0.
\end{align*}

Now we define a function $\omega: \Omega\to \RR$ as follows.
If $p\in \Omega$, let the circular arc $\mathcal{C}_p$ be the connected component of
\[
\Omega\cap \partial B(0, |p|)
\]
containing $p$.

Since $u_1=u_2$ at the endpoints of $\mathcal{C}_p$ and since $\partial u_i / \partial \theta<0$ at all points of $\mathcal{C}_p$, we see that $u_1$ and $u_2$ map $\mathcal{C}_p$ diffeomorphically to the same interval.
Thus there is a unique $q\in \mathcal{C}_p$ such that $u_2(q)=u_1(p)$.
We define
\begin{equation}\label{eqn:omega-p,q-definition}
\omega(p) = \theta(q) -\theta(p).
\end{equation}
Since $u_2>u_1$, $\omega(p)>0$.

Now $\omega$ is a smooth function with the following geometric interpretation.
If $p\in \Omega$, then $\omega(p)$ is the smallest positive angle such that
\[
 \Rr_{\omega(p)}(p,u_1(p)) \in \graph(u_2),
\]
where $\Rr_\theta:\RR^3\to \RR^3$ denotes counterclockwise rotation by $\theta$
about the $z$-axis.

Note that $\omega$ extends continuously to $\overline{\Omega}$ by setting $\omega=0$ on $\partial \Omega$.
By the strong maximum principle, $\omega(\cdot)$ cannot have an interior local maximum.

For $p,q$ as in~\eqref{eqn:omega-p,q-definition}, we have $x(p), x(q) >0$ and $0\le y(p), y(q)\le c$, thus
\[
 \theta(p)\ge 0, \qquad r(p) = r(q).
\]
We conclude that for every $p \in \Omega$ with $\omega(p)= \theta(q) - \theta(p)$,
\begin{align*}
 0 < \omega(p) &\leq \theta(q) = \arcsin( r(q)^{-1}c) = \arcsin ( r(p)^{-1} c).
\end{align*}
Thus $\omega(p)\to 0$ as $r(p)\to\infty$,
so $\omega(\cdot)$ attains its maximum.
But that contradicts the strong maximum principle.
\end{proof}

\section{Uniqueness for Pitchforks}\label{pitchfork-section}

\begin{theorem}\label{pitchfork-theorem}
Let $w\ge \pi$.
Up to an additive constant, there exists a unique translator
\[
f: \RR\times (0,w)\to \RR
\]
such that $f$ has boundary values 
\begin{align*}
f(x,w) &= -\infty,
\\
f(x,0) &=
\begin{cases}
-\infty &\text{if $x<0$}, \\
\infty &\text{if $x>0$},
\end{cases}
\end{align*}
and such that $M:=\graph(f)\cup Z$ is a smooth manifold-with-boundary,
 where $Z$ is the $z$-axis.
\end{theorem}

Here, $M$ is half of a pitchfork,
as in Fig.~\ref{fig:pitchfork} (viewed
 from a point in the first octant).
 The other half, also shown
 in Fig.~\ref{fig:pitchfork},
 is obtained by rotating $M$ by $\pi$ about $Z$.

\begin{remark}\label{pitchfork-remark}
The hypothesis that $M$ is a smooth
manifold-with-boundary is not necessary;
it is implied by the other hypotheses.
See Proposition~\ref{regularity}.
\end{remark}

The hypothesis $w\ge \pi$ is 
necessary: if $w<\pi$, then no such $f$
 exists~\cite{scherk-like}*{Theorem 12.1}.

\begin{proof}
Existence was proved in \cite{scherk-like}*{\S 12}, so it suffices to prove uniqueness.
To prove uniqueness, it will be convenient to add two points $+\Infty$ and $-\Infty$ at infinity
to $\RR^2$ such that, for $p_i\in \RR^2$,
\begin{equation}\label{infinity}
\begin{aligned}
&\text{$p_i\to +\Infty$ if and only if $x(p_i)\to\infty$, and} \\
&\text{$p_i\to-\Infty$ if and only if $x(p_i) \to -\infty$}.
\end{aligned}
\end{equation}

The proof of uniqueness uses the following three facts (see~\cite{scherk-like}*{Theorem 12.1}) about solutions $f$:
\begin{enumerate}[\upshape (F1).]
\item\label{pitchfork-fact-1}
The Gauss map image of $\graph(f)$ depends only on the width $w$.
\item\label{pitchfork-fact-2}
If $\nu(x,y)$ is the upward unit normal to the graph of $f$ at $(x,y,f(x,y))$, then
\[
   \nu(p)\cdot \ee_2 \to 1
\]
as $p\to +\Infty$.
\item\label{pitchfork-fact-3}
There is a strictly decreasing function $\psi:[0,w]\to [\infty,-\infty]$, depending only on $w$, such that
if $x_i\to -\infty$ and $y_i\to y\in [0,w]$, then 
\[
     \pdf{f}y(x_i,y_i) \to \psi(y).
\]
\end{enumerate}

Fact~(F\ref{pitchfork-fact-3}) is based on the fact that the function 
\[
  (x,y) \mapsto f(\bar x + x, y) - f(\bar x,0)
\]
converges to a tilted grim reaper over $\RR\times (0,w)$ as $\bar x\to -\infty$.  
In particular, 
\[
  \psi(y) = \frac{\partial}{\partial y} \left( \left( \frac{w}{\pi} \right)^2 \log(\sin ((\pi/ w)y) - x \sqrt{\frac{w^2}{\pi^2} -1} \right) = \frac{w}{\pi} \frac{\cos( (\pi/w)y)}{\sin((\pi/w)y)}.
\]

Suppose, contrary to the theorem, that there exist two solutions $f_1$ and $f_2$
such that $f_1-f_2$ is not constant.
Let $M_i=\graph(f_i)\cup Z$.

\begin{claim}\label{choosing-claim}
There exists a $v\in \RR\times (-w,w)$ such that the function
\[
g(p) :=  f_1(p) - f_2(p+v).
\]
has a critical point.
Moreover, we can choose $v$ so that $y(v)\ne 0$.
\end{claim}

Of course, the domain of $g$ is the intersection
of the domains of $f_1$ and of $f_2$.

\begin{proof}
The Gauss map images of the graphs of $f_1$ and $f_2$ are the same.
Thus if $q\in \RR\times(0,w)$, then there is a $q'\in \RR\times (0,w)$
such that $Df_2(q')=Df_1(q)$.
Hence, if $v=q'-q$, then the function
\begin{equation}\label{difference-function}
   f_1(p) - f_2(p+v)
   \tag{$\star$}
\end{equation}
has a critical point at $q$.
Note that a critical point
of~\eqref{difference-function}
is equivalent to a point 
 of tangency between the graph of $f_2(p+v)$
 and the foliation whose leaves are vertical translates of $\graph(f_1)$.
It follows,
 by~\cite{morse-rado}*{Corollary~40},
that
the function~\eqref{difference-function}
has a critical point for all $v$ 
sufficiently close to $q'-q$.  
(Indeed, the number of critical points,
 counting multiplicity, 
 of~\eqref{difference-function}
 is a lower-semicontinuous function of $v$.)
In particular,
there exists such a $v$ with $y(v)\ne 0$.
\end{proof}
By relabeling $f_1$ and $f_2$, if necessary, we
can assume
 that $y(v)>0$ in Claim~\ref{choosing-claim}.
Thus
 we have a translator $u_1=f_1$ on $\RR\times(0,w)$,
and a translator 
\[
   u_2: \RR\times (-y(v),w- y(v)) \to \RR
\]
given by $u_2(p)=f_2(p+v)$.
Note that $g:=u_1-u_2$ is defined over the strip 
\[
  \RR\times (0,b),
\]
where $b=w-y(v)$.
Note that $0<b<w$, and that $g$
has boundary values
\begin{align*}
g(x,0) &= -\infty \quad\text{for $x<0$}, \\
g(x,0) &= \infty \quad\text{for $x>0$, and} \\
g(x,b) &= \infty \quad\text{for $x\in \RR$}.
\end{align*}
Also, $g$ has a critical point $p$.  By adding a constant to $f_1$, we can assume that $g(p)=0$.

For $i=1,2$, note that if $\nu_i$ is the upward unit normal to graph of $u_i$,
then
\begin{equation}\label{hypothesis-ok}
 \text{$\nu_i(p)\to \ee_2$ as $p\to +\Infty$},
\end{equation}
by Fact~(F\ref{pitchfork-fact-2}) above.

Let $C$ be a component of the zero set of $g$ that 
contains the critical point $p$.
Then $C$ has the structure of a network in which $p$ is a node, and in which the valence of each node is an even number $\ge 4$.
(See, for example, \cite{colding-minicozzi}*{Theorem~7.3}.)
By the maximum principle, $C$ contains no closed
loops, so it is a tree with at least $4$ ends.
From the boundary values of $g$, we see that each end tends to $-\Infty$, to $+\Infty$, or to $(0,0)$.

Since $M_1:=\graph(u_1)\cup Z$ is a smooth manifold-with-boundary, we see that there is a small disk $D\subset \RR^2$
centered at $(0,0)$ such that $D\cap g^{-1}(0)$ consists
of a single curve $\gamma$ joining $(0,0)$ to a point on $\partial D$.
Thus at most one end of $C$ tends to $(0,0)$.
(There might be no such end since $\gamma$ might lie
 a component of $g^{-1}(0)$ other  than $C$.)

We claim that there is an $\tilde{x}< 0$ such that
\begin{equation}\label{g-increase}
  \pdf{g}{y}(x,y) >  0 
  \quad\text{on $(-\infty,\tilde x] \times (0,b)$}.
\end{equation}
For if not, there would be $(x_i,y_i)$ with $x_i\to -\infty$ such that
\[
   \pdf{g}{y}(x_i,y_i) \le 0.
\]
Passing to a subsequence, we can assume that $y_i\to y\in [0,b]$, and therefore
\begin{align*}
 0 &\ge
   \pdf{g}y(x_i,y_i) \\
   &= \pdf{u_1}y(x_i,y_i) - \pdf{u_2}y(x_i,y_i) \\
   &\to \psi(y) - \psi(y+y(v)) \\
   &> 0,
\end{align*}
since $\psi$ is strictly decreasing, a contradiction.
Thus, there is an $\tilde x$ for 
  which~\eqref{g-increase} holds.
Hence, for $x\le \tilde x$, there is at most one $y$ such that $(x,y)\in C$.
Thus, at most one end of $C$ tends to $-\Infty$.

Since at most one end of $C$ tends to $-\Infty$
and at most one end tends to $(0,0)$, 
at least two ends must tend to $+\Infty$.
Hence, $C$ contains two embedded curves $\alpha$ and $\beta$ that join $p$ to $+\Infty$ and that are disjoint except at $p$.
  Let $\Omega$ be the open region bounded by $\alpha\cup\beta$.
Then $g=0$ on $\partial \Omega$ and $g$ does not vanish anywhere on $\Omega$.
But, by~\eqref{hypothesis-ok}
and the angular maximum principle
(Theorem~\ref{angle-theorem}), $g\equiv 0$ on $\Omega$, a contradiction.
\end{proof}

\section{Uniqueness for Translating Helicoids}\label{helicoid-section}

\begin{theorem}\label{helicoid-theorem}
Let $w< \pi$.  Up to an additive constant, there exists a 
 unique translator 
\[
   f: \RR\times (0,w)\to \RR
\]
with the following property. There exists an $a\in \RR$ such that
$f$ has boundary values
\begin{align*}
 f(x,0) &=
 \begin{cases}
 \infty &\text{if $x<0$}, \\
 -\infty &\text{if $x>0$},
 \end{cases}
 \\
 f(x,w)
 &=
 \begin{cases}
-\infty &\text{if $x<a$}, \\
 \infty &\text{if $x>a$},
 \end{cases}
\end{align*}
and such that $M:=\graph(f)\cup Z \cup L$
is a smooth manifold-with-boundary, 
  where $L$ is the vertical line through $(a,w)$.
\end{theorem}

Here, $M$ is a fundamental piece of a translating 
helicoid, the rest of which is obtained from $M$ by iterated Schwarz reflections,
as in Fig.~\ref{fig:helicoid}.
 The hypothesis $w<\pi$ is 
 necessary: if $w\ge \pi$, then no
 such $f$ 
 exists~\cite{scherk-like}*{Theorem 11.1}.

\begin{remark}\label{helicoid-remark}
The hypothesis that $M$ is a smooth manifold-with-boundary
 is not necessary;
it is implied by the other hypotheses.
See Proposition~\ref{regularity}.
\end{remark}

\begin{proof}
Existence was proved in \cite{scherk-like}*{\S 11}, so it suffices to prove uniqueness.
As proved there, solutions $f$ have the following properties:
\begin{enumerate}[\upshape (P1).]
\item The Gauss map image of the graph of $f$ depends only on $w$.
\item\label{nu-property} $\nu(p)$ converges to $\ee_2$ as $p\to -\Infty$ and to $-\ee_2$ as $p\to +\Infty$.
\end{enumerate}
Here $\pm \widehat\infty$ are the points at $x=\pm \infty$ given by~\eqref{infinity}.

Suppose, contrary to the theorem, that there exist two solutions $f_1$ and $f_2$.
Claim~\ref{choosing-claim} in the proof 
of Theorem~\ref{pitchfork-theorem} also 
holds here 
 (with exactly the same proof).
That is,
 there exists a $v\in \RR\times(0,w)$
with $y(v) \ne 0$ such that the function 
\[
   g(q) = f_1(q) - f_2(q+v)
\]
has a critical point $p$. 
By relabeling, if necessary, we can assume
that $y(v)>0$.
Note that
$g$ is defined over a strip $\RR\times (0,b)$, where $0<b<w$.
We may assume (by adding a constant to $f_1$) that $g(p)=0$.

Note that $g$ has boundary values
\[
g(x,0) 
=
\begin{cases}
\infty &\text{if $x<0$}, \\
-\infty &\text{if $x>0$}
\end{cases}
\]
and
\[
g(x,b) 
=
\begin{cases}
\infty &\text{if $x<a$}, \\
-\infty &\text{if $x>a$}.
\end{cases}
\]

Let $C$ be the connected component of $g^{-1}(0)$ containing $p$.
Exactly as in the proof of Theorem~\ref{pitchfork-theorem}, $C$ is a tree with at least four ends,
and each end tends to $(0,0)$, $(a,b)$, $-\Infty$, or $+\Infty$.
Also, as in that proof, at most one end tends to each of those points.
(In the cases of $\pm \Infty$, this is by
 Property~(P\ref{nu-property}) and
 the angular maximum
 principle, Theorem~\ref{angle-theorem}.)
Thus, exactly one end tends to each of those points. 
Consequently, $C$ contains a curve $C'$ that joins $(0,0)$ to $(a,b)$.

Let $U$ be the component
of $(\RR\times (0,b))\setminus C'$
on which $x$ is bounded below:
\[
  \inf_U x(\cdot) > -\infty.
\]
Since $C$ has an end tending
 to $+\widehat\infty$,
 there are points in $U$ at which $g=0$.
 By the strong maximum principle, there are also points at which $g>0$.
Thus the set
\[
 \Omega := U\cap \{g>0\}
\]
is nonempty.  Now $\partial \Omega$
contains no points on
 the rays $\{(x,0): x>0\}$
 and $\{(x,b): x>a\}$, since $g\equiv -\infty$
on those rays.  Hence $g\equiv 0$ 
on $\partial \Omega$.  But by Property~(P\ref{nu-property}) and
the angular maximum
principle (Theorem~\ref{angle-theorem}), 
$g\equiv 0$
on $\Omega$, a contradiction.
\end{proof}

\section{Nonexistence of The Yeti}\label{yeti-section}

We will use the following facts about tilted grim reaper translators:
\begin{lemma}\label{reaper-lemma}
Given a vector $v\in\RR^2$, there is an $a\ge \pi/2$, a tilted grim reaper
\[
   f: \RR\times (-a,a) \to \RR,
\]
and  a $y\in (-a,a)$ such that
\begin{align*}
D f(x,y) = v
\end{align*}
for all $x$.
If $|y| \le \pi/4$, then $|\frac{\partial f}{\partial y}(x,y)| \le 1$.
\end{lemma}

\begin{proof}
Up to an additive constant, a tilted grim reaper over the indicated strip is given by
\[
  f(x,y)= \kappa^2\log(\cos(y/\kappa)) \pm x\sqrt{\kappa^2 - 1},
\]
where $\kappa= 2a/\pi \ge 1$.
Thus $\frac{\partial f}{\partial x} = \pm \sqrt{\kappa^2-1}$, so we can choose $\kappa$ and the sign of $\pm$ 
so that $\frac{\partial f}{\partial x}\equiv v_1$.
Note that
\[
  \frac{\partial f}{\partial y}(x,y) = -\kappa \tan(y/\kappa),
\]
which ranges from $-\infty$ to $+\infty$, so (given $\kappa$)
we can choose $y$ so that $\frac{\partial f}{\partial y}(x,y)=v_2$.

Now, given $|y| \leq \pi/4$, we know that $\kappa \mapsto |\kappa\tan(y/\kappa)|$ is a decreasing 
 function of $\kappa\in [1,\infty)$, so
\begin{align*}
\left|\frac{\partial f}{\partial y}(x,y)\right|
&=
| \kappa\tan(y/\kappa)| 
\\
&\le
| \tan y|
\\
&\le 
1.
\end{align*}
\end{proof}
By definition, a yeti is (up to translation and rotation) a complete translator $M$ such that $Z\subset M$ and such that $M\setminus Z$ is the graph of a function over $\RR^2\setminus \{y=0\}$.
If $u$ is the restriction of that function to the halfplane $\{y>0\}$, then (by completeness) $u$ has boundary value $+\infty$ on one component of $\{y=0\}\setminus\{(0,0)\}$ and $-\infty$ on the other component.
The following theorem proves that such translators do not exist.
\begin{theorem}\label{yeti-theorem}
There is no solution 
\[
   u: \RR\times (0,\infty)\to \RR
\]
of the translator equation such that
\[
u(x,0)
=
\begin{cases}
\infty &\text{for $x<0$, and} \\
-\infty &\text{for $x>0$},
\end{cases}
\]
and such that $\graph(u)\cup Z$ is a 
  smooth manifold-with-boundary.
\end{theorem}

\begin{remark}\label{yeti-remark}
The hypothesis that $\graph(u)\cup Z$
is a smooth manifold-with-boundary is not necessary.
See Proposition~\ref{regularity}.
\end{remark}

The idea of the proof is as  follows.
We suppose that there is such a 
function $u$, and we take a subsequential limit of
\[
   u(x+s,y) - u(s,1)
\]
as $s\to\infty$.  (Here $1$ is arbitrary; we could use
any positive number.)
We show that the subsequential limit is a complete,
translating graph over all of $\RR\times(0,\infty)$,
which is impossible since, by
the classification~\cite{graphs-himw} of complete graphical translators, no such translator exists.

A priori, the subsequential limit might not exist, or it might
be a graph over a proper subset of
  $\RR\times(0,\infty)$.
To exclude those possibilities, we need suitable gradient 
bounds on $u$.  
By \cite{graphs-himw}*{Theorem 6.7}, to get such bounds, it suffices to show that $|Du|$ is bounded
on some connected, unbounded set contained in 
a halfstrip of the form $[x_1,\infty)\times I$, where $I$
is compact interval in $(0,\infty)$.
Fortunately, we are able to produce 
the necessary set,
namely the graph of a certain smooth
  function $Y:[x_1,\infty) \to I$ such that, 
  on that graph,  $\frac{\partial u}{\partial y}=0$
  and $|\frac{\partial u}{\partial x}|$ is bounded.

\begin{proof}[Proof of Theorem~\ref{yeti-theorem}]
Suppose, contrary to the theorem, that there is such a function $u$.

\setcounter{claim}{0}
\begin{claim}\label{reaper-claim}
Let $a\ge \pi/2$, and let
\[
  f: \RR\times (-a,a)\to \RR
\]
be a function whose graph is a complete translator.
(Thus $f=-\infty$ on the boundary.)
Then the infimum of $u-f$ on the half-strip
\[
     S:=(-\infty,-1]\times (0,a)
\]
is attained at a point in the right edge $\{-1\}\times (0,a)$ of $S$.
\end{claim}

(We only need the claim when $f$ is a tilted grim reaper.)

\begin{proof}
Let $\Rr_\theta:\RR^2\to\RR^2$ be counterclockwise rotation through $\theta$ about the point $(-1,0)$.
For $\theta\in (0,\pi/2)$, let
\[
   S(\theta) := (\Rr_\theta S) \cap \{y>0\}.
\]
Thus $S(\theta)$ is a triangular region with $(-1,0)$ as one vertex.
Let $E(\theta)$ be  the interior of the edge joining $(-1,0)$ to the vertex in $\{y>0\}$,
and let $E(0)$ be its limit as $\theta$ tends $0$:
\[
  E(0) = \{-1\}\times (0,a).
\]
Define
\[
   f_{\theta}: S(\theta) \to \RR
\]
by 
\[
  f_{\theta}(p) = f(\Rr_{-\theta}(p)).
\]

Note that
\[
  u-f_{\theta}: S(\theta) \to \RR
 \]
 has boundary value $+\infty$ at all boundary points except those in the edge $E(\theta)$.
 By the maximum principle, the minimum of $u-f_{\theta}$ is attained at a point
 in $E(\theta)$:
 \begin{equation}\label{rotated}
     \min_{S(\theta)}(u-f_{\theta})  = \min_{E(\theta)}(u-f_{\theta}).
 \end{equation}
The result follows by letting $\theta\to 0$. 
(Note that the right-hand side of~\eqref{rotated} depends
continuously on $\theta$.)
\end{proof}

\begin{claim}\label{c-claim}
There is a $c\in (0,\infty)$ with the following property.
For every $x>0$, there is a $y\in (0,c)$ such that $\frac{\partial u}{\partial y}(x,y)<0$.
\end{claim}

\begin{proof}
Since $u(-1,0)=\infty$, 
\[
  \liminf_{y\to 0} \frac{\partial u}{\partial y}(-1,y) = -\infty.
\]
Thus we can choose a point $p$ with $x(p)=-1$ such that 
\begin{equation}\label{conditions}
\begin{gathered}
\text{$0<y(p)<\pi/4$, and} \\ 
\frac{\partial u}{\partial y}(p) < -1.
\end{gathered}
\end{equation}

Let 
\[
  g: \RR\times (b,c)\to \RR
\]
be the tilted grim reaper surface defined over a strip parallel to the $x$-axis such that
 $g(p)=u(p)$ and $Dg(p)=Du(p)$.
(The function $g$ exists by Lemma~\ref{reaper-lemma}.)
By~\eqref{conditions} and Lemma~\ref{reaper-lemma}, the distance from $p$ to the center line $\{y=(b+c)/2\}$ 
of the strip is greater than $\pi/4$, and thus
\[
 \frac{b+c}2 < y(p) -\frac{\pi}4 < 0.
\]
Consequently,
\begin{equation}\label{down}
\text{$\frac{\partial g}{\partial y}<0$ on $\RR\times [0,c)$.}
\end{equation}
Choose $a$ large enough that $[b,c]\subset (-a,a)$, and let
\[
   f: \RR\times (-a,a) \to \RR
\]
be a tilted grim reaper function such that $\frac{\partial f}{\partial x} < 0$.
By Claim~\ref{reaper-claim}, and by adding a constant to $f$, we can assume that
\[
    \text{$u-f\ge 0$ on $(-\infty,-1]\times (0,c)$}.
\]
Since $[b,c]\subset (-a,a)$, we see that $\frac{\partial f}{\partial x}- \frac{\partial g}{\partial x}$ is a constant $-k<0$.
Thus 
\begin{align*}
f(x,y) - g(x,y) 
&= -kx  + ( f(0,y) - g(0,y))
\\
&\ge -kx + f(0,c) - g(0,0) 
\end{align*} 
since $f(0,y)$ and $g(0,y)$ are decreasing functions of $y\in [0,c]$.
Thus, on $(-\infty,-1]\times[0,c)$,
\begin{align*}
u(x,y) - g(x,y)
&\ge
f(x,y) - g(x,y)  \\
&\ge 
\delta - kx
 \end{align*}
for a constant $\delta$.
Hence
\begin{equation}\label{grows}
  \lim_{x\to -\infty} \min_{y\in [0,c)} (u(x,y) -g(x,y)) = \infty.
\end{equation}

Let $h=u-g$. Thus the domain of $h$ is the strip
\[
  \RR\times (0,c),
\]
and $h$ has boundary values
\[
h(x,0)
=
\begin{cases}
\infty &\text{if $x<0$,} \\
-\infty &\text{if $x>0$,}
\end{cases}
\]
and
\[
h(x,c) = \infty  \quad(x\in \RR).
\]
Let $C$ be the connected component of $h^{-1}(0)$ that
contains $p$.

Since $Dh(p)=0$, $C$ is a tree with at least $4$ ends.
From the boundary values of $h$, we see that each end of $C$
tends to $-\widehat\infty$, to $\widehat\infty$, or to $(0,0)$. 
  Here $\pm \widehat\infty$ are the points at 
  $x=\pm \infty$ 
  given by~\eqref{infinity}.

By~\eqref{grows}, no end tends to $-\widehat\infty$. 

As in the proof of Theorem~\ref{pitchfork-theorem},
 at most one end tends to $(0,0)$.

Thus, 
at least three ends of $C$ tend to $+\Infty$.
Consequently, for each $x>0$, there are
  at least $3$ values of $y\in [0,c)$
for which $h( x,y)=0$.
Let $0<y_1<y_2<c$ be two of those values.
Then
\[
   u( x, y_1) - g( x, y_1) = u( x, y_2) - g( x, y_2)=0.
\]
so
\[
  u( x, y_2) =  g( x, y_2)< g(x,y_1)=u(x, y_1),
\]
since $\frac{\partial g}{\partial y}(x,y)<0$ for $y\in [0,c)$.
Thus there is a $y\in (y_1,y_2)\subset (0,c)$ such that $\frac{\partial u}{\partial y}(x,y)<0$.
\end{proof}

\begin{claim}
Let $c$ be as in Claim~\ref{c-claim}.
Let $Q$ be the set of $(x,y)\in \RR\times(0,c)$ such that $\frac{\partial u}{\partial y}(x,y)=0$.
There is an $x_1\in (0,\infty)$ such that $Q\cap\{x\ge x_1\}$
is the graph of a smooth function 
$
   Y: [x_1,\infty)\to (\pi/4, c)
$.
Furthermore, 
\[
   \sup_{Q\cap \{x \ge x_1\}} |Du|< \infty.
\]
\end{claim}

\begin{proof}
Suppose $p_n\in Q$ and $x(p_n)\to\infty$.
After passing to subsequence, $y(p_n)$ converges to a limit $\hat y\in [0,c]$, 
and
\[
  \graph(u) - (x(p_n),0,u(p_n))
\]
converges to a complete translator that lies in the halfspace $\{y\ge 0\}$.
Let $\widehat M$ be the connected component of the limit  that contains $p:=(0,\hat y,0)$.

Note that the $y$-axis is tangent to $\widehat M$ at $p$.  
Thus, $\widehat M$ cannot be a vertical plane (since $\widehat M$ lies in $\{y\ge 0\}$).
Hence, $\widehat M$ is a graph. 
By the classification of graphical 
 translators~\cite{graphs-himw}, it is the graph of a function $\phi$ over a strip in $\{y\ge 0\}$
 of width $\ge \pi$.
Also, by that classification, 
  $\frac{\partial^2 \phi}{\partial y^2}<0$ at all points, and
$\frac{\partial \phi}{\partial y}=0$ only on the midline of the strip.
Since $\frac{\partial \phi}{\partial y}(0,\hat y)=0$, we see that
 the line $\{y=\hat y\}$ is the midline of the strip.
Since the strip has width $\ge \pi$, 
\[
  \hat y\ge \pi/2.
\]

Thus, we have shown that every sequence $p_n\in Q$ with $x(p_n)\to\infty$ has a subsequence $p_{n(i)}$
such that
\begin{enumerate}
\item $y(p_{n(i)})$ converges to a limit $\hat y\ge \pi/2$.
\item $\frac{\partial^2 u}{\partial y^2}(p_{n(i)})$ converges to a limit
 $\frac{\partial^2 \phi}{\partial y^2}(0,\hat y) <0$.
\item $Du(p_{n(i)})$ converges to a limit $D\phi(0,\hat y)$.
\end{enumerate}
Hence, we can choose $x_1>0$ large enough so that,
for all $p\in Q$ with $x(p)\ge x_1$,
\begin{gather}
  y(p) >\frac{\pi}4, \label{pi/4}   \\
   \frac{\partial^2 u}{\partial y^2}(p) < 0,  \label{bends-down}
\end{gather}
and so that 
\begin{equation}\label{slopey}
   \sup_{Q\cap \{x\ge x_1\}} |Du| < \infty.
\end{equation}

Let $x\ge x_1$.
We claim there is exactly one $y\in (0,c)$ for which $\frac{\partial u}{\partial y}(x,y)=0$, i.e.,
for which $(x,y)\in Q$.

 Now $c$ was chosen according to 
 Claim~\ref{c-claim}.
 Thus there is a $\tilde y\in (0,c)$ for which $\frac{\partial u}{\partial y}(x,\tilde y)<0$.
Since $u(x,0)= -\infty$, there is
at least one value of $y\in (0,\tilde y)$ for which $\frac{\partial u}{\partial y}(x,y)=0$.

If there were two such values $y_1<y_2$,
 then by~\eqref{bends-down}, 
the function 
\[y\in (y_1,y_2)\mapsto u(x,y)
\]
would attain its minimum at some $\hat y\in (y_1,y_2)$.
Thus $\frac{\partial u}{\partial y}(x,\hat y)=0$ and $\frac{\partial^2 u}{\partial y^2}(x, \hat y)\ge 0$, contrary to~\eqref{bends-down}.

Hence there is exactly one $y\in (0,c)$ -- call it $Y(x)$ -- such that $(x,y)\in Q$.
Note that $Y(x)\in (\pi/4,c)$ (by~\eqref{pi/4}),
and 
that $|Du|$ is bounded on the graph of $Y$ 
(by~\eqref{slopey}).
Finally, $Y(x)$ is a smooth function of $x\in [x_1,\infty)$, by~\eqref{bends-down} and the implicit function theorem.
\end{proof}

We are now ready to complete the proof of Theorem~\ref{yeti-theorem}.
We use the following special case of Lemma~6.3 in~\cite{graphs-himw}:

\begin{lemma}\label{slope-bound-lemma}
Suppose $u: [x_1,\infty)\times (0, A)\to \RR$ is a translator.
Suppose $I$ is a compact interval in $(0,A)$ and that
 $Y:[x_1,\infty)\to I$ is continuous function such that $|Du|$ is bounded on the graph of $Y(\cdot)$.
Then for every compact subset $K$ of $(0,A)$,
\begin{equation}\label{slope-bound}
  \sup_{[x_1,\infty)\times K} |Du| < \infty.
\end{equation}
\end{lemma}

We have shown that, in our situation, the hypotheses of Lemma~\ref{slope-bound-lemma} hold with $A=\infty$ and $I=[\pi/4,c]$.
Hence,~\eqref{slope-bound} holds for every compact $K\subset (0,\infty)$.
It follows that the functions
\[
   (x,y)\in \RR\times (0,\infty) \mapsto u(x + s, y) - u(s,1)
\]
converge subsequentially as $s\to\infty$ to a function
\[
   w: \RR\times (0,\infty)\to \RR
\]
whose graph is a complete translator. 
But there is no such function by \cite{graphs-himw}*{Theorem 6.7}.
\end{proof}

The following proposition was not used
in this paper, but it justifies the assertions
in Remarks~\ref{pitchfork-remark}, 
 \ref{helicoid-remark}, and~\ref{yeti-remark}.

\begin{proposition}\label{regularity}
Suppose that
\[
  u:\{(x,y)\in B^2(0,R): y>0\}\to \RR
\]
is a smooth solution to the translator equation
such that 
\[
   u(x,0)
   =
   \begin{cases}
       -\infty &\text{if $x<0$}, \\
       \infty &\text{if $x>0$}.
   \end{cases}
\]
Then $\graph(u)\cup Z$
is a smooth-manifold-with-boundary.
\end{proposition}

\begin{proof}
Let $C(R)$ be the cylinder $B^2(0,R)\times\RR$.
Let $M$ be the graph of $u$.
Then $M$ minimizes the translator area in the following sense:
if $K$ is a compact portion of $M$
and if $K'$ is a compact surface in $C(R)\cap \{y>0\}$
with $\partial K'=\partial K$,
then
\[
  \area(K)\le \area(K').
\]
Here and in what follows, we use $\on{area}(\Sigma)$ to refer to the translator area of a surface $\Sigma$, computed with respect to Ilmanen's translator metric $g_{ij}(x,y,z)=e^{-z}\delta_{ij}$.

Actually, the assumption that $K'$ lies in $C(R) \cap \{ y > 0 \}$ is not needed.
For suppose that $K'$ is any compact surface with $\partial K'=\partial K$.
Let $Q$ be a compact convex subset of 
$B^2(0,R)\cap\{y>0\}$ such that $Q\times \RR$
contains $K$ and let $K''$ be the projection of
$K'$ to $Q\times \RR$.  Then
\[
   \area(K)\le \area(K'')\le \area(K').
\]

In particular, if $K$ is a convex, compact subset of 
$C(R)\cap\{y>0\}$, then $M\cap K$
and $V\cap \partial K$ have the same boundary,
where
\[
  V = \{(x,y,z)\in C(R): y>0,\, z<u(x,y)\}
\]
is the subgraph of $u$.
Thus
\begin{align*}
\area(M\cap K)
&\le
\area(V\cap \partial K) 
\\
&\le
\area(\partial K)
\end{align*}
Thus, if $K$ is any compact, convex subset of $C(R)$, then
\begin{align*}
  \area(M\cap K\cap\{y\ge \eps\})
  &\le
   \area(\partial K\cap\{y\ge \eps\}) \\
  &\le  \area(\partial K).
\end{align*}
Thus, letting $\eps\to 0$,
\[
\area(M\cap K)\le \area(\partial K).
\]

The fact that $M\cap K$ has finite area
together with the fact that $M\cap\{y=0\}$
has area $0$ implies that the surfaces
\begin{equation}\label{almost}
   M\cap K\cap \{y>\eps\} \tag{$\star$}
\end{equation}
converge (as currents, with the mass norm)
to $M\cap K$ as $\eps\to 0$.  Thus, $M\cap K$ is area-minimizing,
since it is the limit of the area-minimizing
surfaces~\eqref{almost}.

Let $S$ be the strip
\[
  S= \{(x,0,z): x\in (0,R), \, z\in \RR\}.
\]
Let $[\![M]\!]$, $[\![S]\!]$, and $[\![V]\!]$ be the locally integral currents
obtained by orienting $M$ by the upward unit normal, 
and $S$ by the unit normal $\ee_2$, 
and by giving $V$ the standard orientation.

Then
\[
  [\![\partial V ]\!] \mres C(R) = [\![M]\!] - [\![S]\!],
\]
so
\begin{align*}
[\![ \partial M ]\!] \mres C(R) 
&= [\![ \partial S ]\!] \mres C(R) \\
&= [\![Z ]\!].
\end{align*}
We have therefore proved that $M$ is translator-area minimizing and has boundary (in $C(R)$) the line $Z$
with multiplicity~$1$.  
Hence by~\cite{hardt-simon}, $M\cup Z$ is a smooth 
manifold-with-boundary.
\end{proof}

\end{document}